\documentclass[11pt]{amsart}
\usepackage{amscd}
\usepackage{amsmath, amssymb, mathrsfs, verbatim, multirow}
\usepackage[all]{xy}
\usepackage{pifont}
\usepackage{float}
\usepackage{color}

\def\NZQ{\mathbb}               % the font for N,Z,Q,R,C
\def\NN{{\NZQ N}}
\def\QQ{{\NZQ Q}}
\def\ZZ{{\NZQ Z}}

\newtheorem{Theorem}{Theorem}[section]
\newtheorem{Lemma}[Theorem]{Lemma}

\newtheorem{Proposition}[Theorem]{Proposition}
\newtheorem{Remark}[Theorem]{Remark}

\newtheorem{Question}[Theorem]{Question}
\let\epsilon\varepsilon
\let\phi=\varphi
\let\kappa=\varkappa

\def \a {\alpha}

\begin{document}
\title[Essentially generation and invariants]{Essential finite generation of valuation rings in characteristic zero algebraic 
function fields}
\author{Steven Dale Cutkosky}
\thanks{Steven Dale Cutkosky was partially supported by NSF grant DMS-1700046.}
\begin{abstract}
Let $K$ be a characteristic zero algebraic function field with a valuation $\nu$. Let $L$ be  a finite extension of $K$ and $\omega$ be an extension of $\nu$ to $L$. We establish that the  valuation ring $V_{\omega}$ of $\omega$ is essentially finitely generated over the valuation ring $V_{\nu}$ of $\nu$ if  and only if the initial index $\epsilon(\omega|\nu)$ is equal to the ramification index $e(\omega|\nu)$ of the extension. This gives a positive answer, for characteristic zero algebraic function fields, to a question posed by Hagen Knaf. 
\end{abstract}

%\keywords{Local uniformization, \textcolor{red}{essentially finitely} generated, invariants of valuations}
%\subjclass[2010]{Primary 13A18}

\maketitle

\section{Introduction}\label{SecInt} Suppose that $K$ is a  field and $\nu$ is a valuation of $K$. Let  $V_{\nu}$ be the valuation ring of $\nu$ with maximal ideal $m_{\nu}$  and $\Gamma_{\nu}$ be the value group of $\nu$. Suppose that $K\rightarrow L$ is a finite field extension and $\omega$ is an extension of $\nu$ to $L$. We have associated ramification and inertia indices of the extension $\omega$ over $\nu$
$$
e(\omega|\nu)=[\Gamma_{\omega}:\Gamma_{\nu}]\mbox{ and }
f(\omega|\nu)=[V_{\omega}/m_{\omega}:V_{\nu}/m_{\nu}].
$$
The defect of the extension of $\omega$ over $\nu$ is 
$$
d(\omega|\nu)=\frac{[L^h:K^h]}{e(\omega|\nu)f(\omega|\nu)}
$$
where $K^h$ and $L^h$ are  henselizations of the valued fields $K$ and $L$.
This is a positive integer (as shown in \cite{EP}) which is 1 if $V_{\nu}/m_{\nu}$ has characteristic zero and is a power of $p$ if $V_{\nu}/m_{\nu}$ has positive characteristic $p$.

Let $H$ be an ordered subgroup of an ordered abelian group $G$. The initial index $\epsilon(G|H)$ of $H$ in $G$ is defined (\cite[page 138]{End}) as 
$$
\epsilon(G|H)=|\{g\in G_{\ge 0}\mid g<H_{>0}\}|,
$$
where 
$$
G_{\ge 0}=\{g\in G\mid g\ge 0\}\mbox{ and }H_{>0}=\{h\in H\mid h>0\}.
$$
We define the initial index $\epsilon(\omega|\nu)$ of the extension as $\epsilon(\Gamma_{\omega}:\Gamma_{\nu})$. 

We always have that $\epsilon(\omega|\nu)\le e(\omega|\nu)$ (\cite[(18.3)]{End}).

If $S$ is a subsemigroup of an abelian semigroup $T$, we say that $T$ is a finitely generated $S$-module if there exists a finite number of elements $g_1,\ldots,g_t\in T$ such that 
$$
T=\cup_{i-1}^t(g_i+S).
$$
It is shown in \cite[Proposition 3.3]{CN}  that $\epsilon(\omega|\nu)=e(\omega|\nu)$ if and only if $(\Gamma_{\omega})_{\ge 0}$ is a finitely generated $(\Gamma_{\nu})_{\ge 0}$-module. We remark that  $(\Gamma_{\nu})_{\ge 0}$ is the semigroup of values of elements of the valuation ring $V_{\nu}$.
 
Let $D(\nu,L)$ be the integral closure of $V_{\nu}$ in $L$. The localizations of $D(\nu,L)$ at its maximal ideals are the valuation rings $V_{\omega_i}$ of the extensions $\omega_i$ of $\nu$ to $L$. We have the following remarkable theorem.

\begin{Theorem}(\cite[Theorem 18.6]{End}\label{ThmEnd}) The ring $D(\nu,L)$ is a finite $V_{\nu}$-module if and only if 
$$
d(\omega_i|\nu)=1\mbox{ and }\epsilon(\omega_i|\nu)=e(\omega_i|\nu)
$$
for all extensions $\omega_i$ of $\nu$ to $L$.
\end{Theorem}

An equivalent formulation is given in \cite[Th\'eor\`eme 2,  page 143]{NB}. 

Suppose that $A$ is a subring of a ring $B$. We will say that $B$ is essentially finitely generated over $A$ (or that $B$ is essentially of finite type over $A$) if $B$ is  a localization of a finitely generated $A$-algebra. 

Hagen Knaf proposed the following interesting question, asking for  a local form of the above theorem.

\begin{Question}(Knaf) Suppose that   $\omega$ is an extension of $\nu$ to $L$. Is $V_{\omega}$ essentially finitely generated over $V_{\nu}$ if and only if 
$$
d(\omega|\nu)=1\mbox{ and }\epsilon(\omega|\nu)=e(\omega|\nu)?
$$
\end{Question}

Knaf proved the implies direction of his question; his proof is reproduced in \cite[Theorem 4.1]{CN}.

If $e(\omega|\nu)=1$, $d(\omega|\nu)=1$ and $V_{\omega}/m_{\omega}$ is separable over $V_{\nu}/m_{\nu}$,
then the only if direction of the question is true, as is proven in \cite{Hens}. Also, the only if direction of the question is true if $L/K$ is normal or $\omega$ is the unique extension of $\nu$ to $L$ by \cite[Corollary 2.2]{CN}.

The only if direction of the question is proven when $K$ is the quotient field of an excellent two-dimensional excellent local domain and $\nu$ dominates $R$ in \cite[Theorem 1.4]{CN}. The only if direction is proven when $K$ is an algebraic function field over a field $k$, $\nu$ is an Abhyankar valuation of $K$ and 
$V_{\omega}/m_{\omega}$ is separable over $k$ in \cite[Theorem 1.5]{CN}.  

The proof of \cite[Theorem 1.4]{CN} uses the existence of a resolution of excellent surface singularities (\cite{Lip} or \cite{CJS}) and local monomialization of defectless extensions of two dimensional excellent local domains (\cite[Theorem 3.7]{Ramif} and \cite[Theorem 7.3]{CP}). The proof of \cite[Theorem 1.5]{CN} uses the local uniformization theorem for Abhyankar valuations in algebraic function fields of Knaf and Kuhlmann in \cite{KK}.

In this paper, we give a positive answer to the question for characteristic zero algebraic function fields, as stated in  the following theorem. 

\begin{Theorem}\label{TheoremA} Let $K$ be an algebraic function field over a field $k$ of characteristic zero and let $\nu$ be a valuation of $K/k$ ($\nu$ is trivial on $k$). Assume that $L$ is a finite extension of $K$ and $\omega$ is an extension of $\nu$ to $L$. Then $V_{\omega}$ is essentially finitely generated over $V_{\nu}$
if and only if $e(\omega|\nu)=\epsilon(\omega|\nu)$.
\end{Theorem} 

Recall that the defect $d(\omega|\nu)$ must be 1 under an extension of equicharacteristic zero valuation rings, as occurs in Theorem \ref{TheoremA}.

The proof of Theorem \ref{TheoremA} uses an  explicit form of embedded local resolution of singularities  along a valuation in characteristic zero algebraic function fields, by Zariski \cite{Z1} for  rank 1 valuations and as extended to higher rank valuations by ElHitti in \cite{E1}. It also uses the existence of a local monomialization of regular algebraic regular local rings $R\rightarrow S$ of $K$ and $L$ respectively which are dominated by $\omega$ as  shown in \cite{Ast}. It is shown in the proof of Theorem \ref{TheoremA}, that if $e(\omega|\nu)=\epsilon(\omega|\nu)$, then there exists a locally monomial  extension $R\rightarrow S$  along $\omega$ such that if $S$ is a localization of a finitely generated $R$-algebra $F[z_1,\ldots,z_n]$, then $V_{\omega}$ is a localization of the finitely generated $V_{\nu}$-algebra $V_{\nu}[z_1,\ldots,z_r]$. 

It is shown in \cite{C2} that local monomialization is false in positive characteristic, even in dimension two. However, local monomialization is true for defectless extensions in dimension two (\cite[Theorem 3.7]{Ramif} and \cite[Theorem 7.3]{CP}).

\section{Preliminaries and Notation}\label{SecPrem}
We will denote the non-negative integers by $\NN$ and $\ZZ_{>0}$ will denote the positive integers. We will denote the maximal ideal of a local ring $R$ by $m_R$.
If $R$ and $S$ are local rings such that $R$ is a subring of $S$ and $m_S\cap R=m_R$ then we say that $S$ dominates $R$. If $A$ is a domain then ${\rm QF}(A)$ will denote the quotient field of $A$.

Suppose that $A$ is a subring of a ring $B$. We will say that $B$ is essentially finitely generated over $A$ (or that $B$ is essentially of finite type over $A$) if $B$ is  a localization of a finitely generated $A$-algebra. 

We refer to \cite{ZS2} and \cite{End} for basic facts about valuations.

Suppose that $k$ is a field  and $K/k$ is an algebraic function field over $k$. 
An algebraic local ring of $K$ is a local domain which is essentially of finite type over $k$ and whose quotient field is $K$. A birational extension $R\rightarrow R_1$ of an algebraic local ring $R$ of $K$ is an algebraic local ring $R_1$ of $K$ such that $R_1$ dominates $R$.

Suppose that $\nu$ is a valuation of $K/k$ (a valuation of $K$ which is trivial on $k$). Let $V_{\nu}$ be the valuation ring of $\nu$, with maximal ideal $m_{\nu}$. If $A$ is a subring of $V_{\nu}$. then we write $A_{\nu}=A_{m_{\nu}\cap A}$. If $A$ is a local ring which is a subring of $V_{\nu}$ and $m_{\nu}\cap A=m_A$ then we say that $\nu$ dominates $A$. 

Let $u=\mbox{rank }\nu$ and let 
$$
0= P_{\nu,u+1}\subset \cdots\subset P_{\nu,1}=m_{\nu}
$$
be the chain of prime ideals in $V_{\nu}$. Let $\Gamma_{\nu}$ be the valuation group of $\nu$ with chain of convex subgroups
$$
0=\Gamma_{\nu,0}\subset \Gamma_{\nu,1}\subset\cdots\subset \Gamma_{\nu,u}=\Gamma_{\nu}.
$$
Let $s_i$ be the rational rank of $\Gamma_{\nu,{i}}/\Gamma_{\nu,i-1}$ for $1\le i\le u$. For $1\le i\le u$, let $\nu_i$ be the valuation ring $V_{P_{\nu,i}}$ obtained by specialization of $\nu$. In particular, $\nu_1=\nu$. The value group of $\nu_i$ is $\Gamma_{\nu}/\Gamma_{\nu,i-1}$.

Suppose that $T$ is an algebraic local ring of $K$ which is dominated by $\nu$. Define prime ideals
$P_{T,i}=P_{\nu,i}\cap T$ in $T$ for $1\le i\le u$.

The technical condition (A) is defined in 
%Definition 4.2.1 \cite{E}.
\cite[Definition 4.1]{E1}.

\begin{Lemma}\label{Lemma1} Suppose that $\mbox{char}(k)=0$. Then there exists an algebraic regular local ring $T$ of $K$ which is dominated by $\nu$ and such that
\begin{enumerate}
\item[1)] $\mbox{trdeg}_{ {\rm QF}(T/P_{T,i})}{\rm QF}(V/P_{\nu,i}) =0$ for $1\le i\le u$.
\item[2)] $T_{P_{T,i}}$ satisfies condition (A) for $1\le i\le u$.
\end{enumerate}
Further, if $T\rightarrow T_1$ is a birational extension along $\nu$ then $T_1$ satisfies 1) and 2).
\end{Lemma}

\begin{proof} Let $A$ be an algebraic local ring of $K$ which is dominated by $\nu$. Let $z_{ij}\in V_{\nu}$ be such that $\{z_{ij}\}_j$ for $1\le i\le u$ is a transcendence basis of ${\rm QF}(V/P_{\nu,i})$ over 
${\rm QF}(A/P_{A,i})$. This is a finite set. Let $B=A[z_{ij}]_{\nu}$. Then $B$ satisfies 1) and if $B\rightarrow R$ is a birational extension along $\nu$ then $R$ satisfies 1).

By Theorem 6.3 \cite{CG}, for $1\le i\le u$, there exist birational extensions $B_{P_{B,i}}\rightarrow C_i$ such that $C_i$ is dominated by $\nu_i$ and condition (A) holds  for any  algebraic normal local ring of 
$K$ which dominates $C_i$ and is dominated by $\nu_i$. Let $T$ be any regular algebraic local ring of $K$ which is dominated by $\nu$ such that $T$ dominates $B$ 
 and such that $T_{P_{T,i}}$ dominates $C_i$ for all $i$. Then $T$ satisfies the conclusions of the lemma.
\end{proof}

Suppose that $T$ satisfies the conclusions of Lemma \ref{Lemma1}. Suppose that
\begin{equation}\label{eq1}
x_{1,1},\ldots,x_{1,s_1},x_{1,s_1+1},\ldots,x_{1,t_1},x_{2,1},\ldots,x_{2,s_2},x_{2,s_2+1},\ldots,x_{2,t_2},x_{3,1},\ldots, x_{u,t_u}
\end{equation}
are regular parameters in $T$. The regular parameters (\ref{eq1}) are called {\it good parameters} if
$x_{i,1},\ldots,x_{i,t_i}\in P_{T,i}\setminus P_{T,i+1}$ and $\nu(x_{i,1}),\ldots,\nu(x_{i,s_i})$ form a rational basis of $(\Gamma_{\nu,i}/\Gamma_{\nu,i-1})\otimes\QQ$ for $1\le i\le u$.  If $S$ is a subset of $\{1,\ldots,u\}$ then the regular parameters (\ref{eq1}) are called $S$-{\it good parameters} if they are good parameters and $P_{T,i}=(x_{i,1},\ldots,x_{i,t_i},\ldots)$ for $i\in S$. We will say that the parameters (\ref{eq1}) are {\it very good} if they are $\{1,2,\ldots,u\}$-good. We remark that good parameters are always $\{1\}$-good.

Suppose that (\ref{eq1}) are good parameters and
\begin{equation}\label{eq2}
\overline x_{1,1},\ldots,\overline x_{1,s_1},\overline x_{1,s_1+1},\ldots,\overline x_{1,\overline t_1},\overline x_{2,1},\ldots,\overline x_{2,s_2},\overline x_{2,s_2+1},\ldots,\overline x_{2,\overline t_2},\overline x_{3,1},\ldots, \overline x_{u,\overline t_u}
\end{equation}
is another system of parameters in $T$. It is not required that the numbers $t_i$ and $\overline t_i$ are the same. The system of regular parameters (\ref{eq2}) is called an $S$-{\it good change of parameters} if the parameters (\ref{eq2}) are $S$-good and $\overline x_{ij}=x_{ij}$ for $1\le i\le u$ and $1\le j\le s_i$.

%The following lemma follows from embedded local uniformization (it is for instance an immediate consequence of Theorem 6.6 \cite{S}).

%\begin{Lemma}\label{Lemma2}  There exists an algebraic local ring $T$ of $K$ which is dominated by $\nu$ such that the conditions of Lemma \ref{Lemma1} hold for $T$ and $T$ has a system of very good parameters. 
\section{Perron Transforms}

\subsection{Perron transforms of types (1,m), (2,m)  and (3,m)}

The basic Perron transforms of types (1,1) and (2,1) are defined by Zariski in \cite{Z1} for rank 1 valuations. They are used in \cite{Ast} and \cite{C1} to prove local monomialization of morphisms. The Perron transforms of types (1,m), (2,m) and (3,m), for use in higher rank, are defined by ElHitti in \cite{E1}. The notation (1,m), (1,m,r) and (2,m) used in \cite{E1} is a little different from our notation. 

We use the notation of Section \ref{SecPrem} and assume that $k$ has characteristic zero. 

Suppose that $T$ is an algebraic local ring of $K$ which is dominated by $\nu$ and that $T$ satisfies the conclusions of Lemma \ref{Lemma1}.
%This is a summary of the material from Sections 5.1, 5.2  and 5.3 \cite{E}. To mesh well with our indexing of parameters, our m is one more than the m of \cite{E}.
Suppose that (\ref{eq1}) are $S$-good parameters in $T$ and $1\le m\le u$.
We   define a Perron Transform $T\rightarrow T_1$ of type $(1,m)$  along $\nu$.  We first define $N_j$ 
by
$$
x_{mj}=N_1^{a_{j1}}\cdots N_{s_m}^{a_{js_m}}\mbox{ for  $1\le j\le s_m$}
$$
where $a_{ij}\in \NN$ are defined by Perron's algorithm, as explained in Sections B I and B II of \cite{Z1}. We have that  $\mbox{Det}(a_{ij})=\pm 1$ and $\nu(N_j)>0$ for all $j$.

We define  $T_1=T[N_1,\ldots,N_{s_m}]_{\nu}$, which  is a regular local ring. 
We define regular parameters
$\{\overline x(1)_{ij}\}$ in $T_1$ by  
$$
\overline x(1)_{ij}=
\left\{\begin{array}{ll}
N_j&\mbox{ if $i=m$ and $1\le j\le s_m$}\\
x_{ij}&\mbox{ otherwise}
\end{array}\right.
$$
The regular parameters $\{\overline x(1)_{ij}\}$ are $S$-good parameters in $T_1$.

We  now define a Perron Transform $T\rightarrow T_1$ of type $(2,m)$  along $\nu$. This is a generalization of the Perron transform constructed in Section B III of \cite{Z1}. Let $r$ be such that   $s_m<r\le t_m$. We first define $N_j$    by
$$
x_{mj}=\left\{\begin{array}{ll}
N_1^{a_{j1}}\cdots N_{s_m}^{a_{js_m}}N_r^{a_{j,s_m+1}}&\mbox{ if $1\le j\le s_m$}\\
N_1^{a_{s_m+1,1}}\cdots N_{s_m}^{a_{s_m+1,s_m}}N_r^{a_{s_m+1,s_m+1}}&\mbox{ if $ j= r$}
\end{array}\right.
$$
where $a_{ij}\in \NN$, $\mbox{Det}(a_{ij})=\pm 1$ and $\nu(N_1),\ldots,\nu(N_{s_m})>0$ and $\nu_m(N_r)=0, \nu(N_r) \ge 0$.

$N_1,\ldots,N_{s_m},N_r$ satisfying the above conditions always exists, as follows from a small variation in Zariski's algorithm in \cite{Z1}. We construct the Perron transform of Zariski from $x_{m,1},\ldots,x_{m,s_m}$ and $x_{m,r}$ for $\nu_m$, as contructed in Section B III of \cite{Z1}. In this algorithm, the next to last step constructs $M_1,\ldots,M_{s_m},M_r$ and a $(s_m+1)\times(s_m+1)$-matrix $(b_{ij})$ such that
$$
x_{mj}=\left\{\begin{array}{ll}
M_1^{b_{j1}}\cdots M_{s_m}^{b_{js_m}}M_r^{b_{j,s_m+1}}&\mbox{ if  $1\le j\le s_m$}\\
M_1^{b_{s_m+1,1}}\cdots M_{s_m}^{b_{s_m+1,s_m}}M_r^{b_{s_m+1,s_m+1}}&\mbox{ if  $ j= r$}
\end{array}\right.
$$
where $b_{ij}\in \NN$, $\mbox{Det}(b_{ij})=\pm 1$, $\nu_m(M_1),\ldots,\nu_m(M_{s_m}),\nu_m(M_r)>0$ and $\nu_m(M_r)=\nu_m(M_1)>0$. We then have that
$$
\nu_m\left(\frac{M_r}{M_1}\right)=\nu_m\left(\frac{M_1}{M_r}\right)=0.
$$
If $\nu(\frac{M_r}{M_1})\ge 0$, define $N_1,\ldots,N_{s_m},N_r$ by
$$
M_i=\left\{\begin{array}{ll} N_i&\mbox{ if $i\ne r$}\\
N_rN_1&\mbox{ if $i=r$}
\end{array}\right.
$$
If $\nu(\frac{M_1}{M_r})> 0$, define $N_1,\ldots,N_{s_m},N_r$ by
$$
M_i=\left\{\begin{array}{ll} N_1N_r&\mbox{ if $i=1$}\\
M_i=N_i&\mbox{ if $i\ne 1$ and  $i\ne r$}\\
N_1&\mbox{ if $i=r$}
\end{array}\right.
$$

We define $T_1=T[N_1,\ldots,N_{s_m},N_r]_{\nu}$, which  is a regular local ring. Let $\mathfrak m = m_{\nu}\cap T[N_1,\ldots,N_{s_m},N_r]$. Choose  $y\in T_1$  such that
$y$ is the lift to $T_1$ of a generator of the maximal ideal of 
$$
\begin{array}{ll}
&T_1/(x_{1,1},\ldots,x_{1,t_1},\ldots,x_{m-1,t_{m-1}},N_1,\ldots,N_{s_m},x_{m,s_m+1},\ldots, x_{m,r-1},x_{m,r+1},\ldots )\\
\cong &(T/m_T)[N_r]_{\mathfrak m(T/m_T)[N_r]}.
\end{array}
$$
Then
$$
y,x_{1,1},\ldots,x_{1,t_1},\ldots,x_{m-1,t_{m-1}},N_1,\ldots,N_{s_m},x_{m,s_m+1},\ldots, x_{m,r-1},x_{m,r+1},\ldots
$$
is a system of regular parameters in $T_1$. There is a smallest natural number $\lambda$ such that $y\in P_{T_1,\lambda}\setminus P_{T_1,\lambda+1}$.  We define regular parameters $\{\overline x(1)_{ij}\}$ in $T_1$ by

$$
 \overline x(1)_{ij}=\left\{\begin{array}{ll}
 N_j&\mbox{ if $i=m$ and $1\le j\le s_m$}\\
 y&\mbox{ if  $i=m$ and $j=r$}\\
 x_{ij}&\mbox{ otherwise}
 \end{array}\right.
 $$
 if $\lambda= m$, and

 $$
 \overline x(1)_{ij}=\left\{\begin{array}{ll}
 N_j&\mbox{ if $i=m$ and $1\le j\le s_m$}\\
 y&\mbox{ if $i=\lambda$ and $j=t_{\lambda}+1$}\\
 x_{i,j-1}&\mbox{ if  $i=m$ and $j\ge r+1$}\\
 x_{ij}&\mbox{ otherwise}
 \end{array}\right.
 $$
 if $\lambda\ne m$. 
 
 If $i\in S$ and $i>m$, then 
$P_{T,i}T_1=(x_{i,1},\ldots,x_{i,2},\ldots)T_1$ is a regular prime of $T_1$ which has the same height as $P_{T_1,i}$ since $T_1$ satisfies 1) of Lemma \ref{Lemma1}. Since this prime ideal is contained in $P_{T_1,i}$ we have that $P_{T_1,i}=P_{T,i}T_1$.
Thus if $m+1\in S$, then $\lambda\le m$.

 We have that $\{\overline x(1)_{ij}\}$
 are $S'$-good parameters in $T_1$,  
where 
$$
S'=\{j\in S|j>m\}.
$$

We now define a Perron transformation of type (3,m). Suppose that $d_1,\ldots,d_{s_m}\in\NN$, $k>m$ and $1\le l\le t_k$. Let 
$$
N=\frac{x_{kl}}{x_{m,1}^{d_1}\cdots x_{m,s_m}^{d_{s_m}}}.
$$
Then $\nu_k(N)>0$. Let $T_1=T[N]_{\nu}$, which is a regular local ring. Let 
$$
\overline x(1)_{ij}=\left\{\begin{array}{ll}
N&\mbox{ if $i=k$ and $j=l$}\\
x_{ij}&\mbox{ otherwise.}
\end{array}\right.
$$
Then $\{\overline x(1)_{i,j}\}$ are $S$-good parameters in $T_1$.

We will find the following proposition useful.

\begin{Proposition}\label{PropZa} Suppose that $R$ is an algebraic regular local ring of $K$ which is dominated by $\nu$ and $\{x_{ij}\}$ are good parameters in $R$. Suppose that 
$$
M_1=x_{1,1}^{a_{1,1}}\cdots x_{1,s_1}^{a_{1,s_1}}x_{2,1}^{a_{2,1}}\cdots x_{2,s_2}^{a_{2,s_2}}x_{3,1}^{a_{3,1}}\cdots x_{u,s_u}^{a_{u,s_u}}
$$
and
$$
M_2=x_{1,1}^{b_{1,1}}\cdots x_{1,s_1}^{b_{1,s_1}}x_{2,1}^{b_{2,1}}\cdots x_{2,s_2}^{b_{2,s_2}}x_{3,1}^{b_{3,1}}\cdots x_{u,s_u}^{b_{u,s_u}}
$$
 are monomials such that $\nu(M_1)\le\nu(M_2)$. Then there exists a sequence of  Perron tranforms of types (1,m) and (3,m)  along $\nu$, 
 $$
R\rightarrow R_1\rightarrow \cdots\rightarrow R_s,
$$
such that $M_1$ divides $M_2$ in $R_s$. 

If $\nu(M_1)=\nu(M_2)$, then $M_1=M_2$ since the members of 
$$
\{\nu(x_{ij})|1\le i\le u, 1\le j\le s_i\}
$$
are rationally independent.
\end{Proposition}

\begin{proof} Suppose that $\nu(M_1)<\nu(M_2)$. There exists a largest index $l$ such that $\prod_jx_{l,j}^{a_{l,j}}\ne \prod_jx_{l,j}^{b_{l,j}}$.
Then $\nu(\prod_jx_{l,j}^{a_{l,j}})<\nu(\prod_jx_{l,j}^{b_{l,j}})$.  By \cite[Theorem 2]{Z1}, there exists a sequence of Perron transforms of type (1,l)  $R\rightarrow R_1$ along $\nu$ such that $\prod_jx_{l,j}^{a_{l,j}}$ divides $\prod_jx_{l,j}^{b_{l,j}}$ in $R_1$. Writing $M_1$ and $M_2$ in the regular parameters $\{z(1)_{i,j}\}$ of $R_1$ as
$$
M_1=\prod x(1)_{i,j}^{a(1)_{i,j}}\mbox{ and }M_2=\prod x(1)_{i,j}^{b(1)_{i,j}},
$$
where the product is over $1\le i\le u$ amd $1\le j\le s_i$,
we have that 
$$
M_2=\left(\prod_{i<l,j}x(1)_{i,j}^{b(1)_{i,j}}\right)\left(\prod_j x(1)_{l,j}^{b(1)_{l,j}}\right)
\left(\prod_{i>l,j}x(1)_{i,j}^{a(1)_{i,j}}\right)
$$
with $b(1)_{l,j}-a(1)_{l,j}\ge 0$ for all $j$ and for some $j$, $b(1)_{l,j}-a(1)_{l,j}> 0$.
Without loss of generality, this occurs for $j=1$. (If $b(1)_{l,j}=a(1)_{l,j}$ for all $j$, then $a_{l,j}=b_{l,j}$ for all $j$ in contradiction to our choice of $l$.)

Now perform a  sequence of Perron transforms of type (3,n) for $1\le n<l$, $R_1\rightarrow R_m$ along $\nu$ defined by
$x(t)_{l,1}=x(t+1)_{l,1}x(t+1)_{i,j}$ for $i<l$ and $j$ such that $b(t)_{i,j}<a(t)_{i,j}$ where 
$$
M_1=\prod x(t)_{i,j}^{a(t)_{i,j}}\mbox{ and }M_2=\prod x(t)_{i,j}^{b(t)_{i,j}}
$$
to achieve that $M_1$ divides $M_2$ in $R_m$.
\end{proof}

\subsection{Sequences of good monoidal transform sequences}

We now define a {\it good monoidal transform sequence} along $\nu$, which will be abbreviated as a GMTS.
Suppose that $T$ satisfies the conditions of Lemma \ref{Lemma1} and 
$\{\overline x_{i,j}\}$  are good parameters in $T$. Let 
$\{x_{i,j}\}$ be a good change of parameters in $T$. Let $T\rightarrow T_1$
be a Perron transform of one of the types (1,m), (2,m) or (3,m) of the previous  subsection, giving good parameters
$\{\overline x(1)_{i,j}\}$  in $T_1$. Then we call $T\rightarrow T_1$, with the parameters $\{\overline x_{i,j}\}$  and good change of parameters $\{x_{i,j}\}$ in $T$ and good parameters
$\{\overline x(1)_{i,j}\}$ in $T_1$ a good monoidal transform sequence. 

Suppose that $T(0)$ satisfies the conditions of Lemma \ref{Lemma1} and 
$\{\overline x(0)_{i,j}\}$  are good parameters in $T$.
A sequence of GMTSs is a sequence
$$
T(0)\rightarrow T(1)\rightarrow \cdots \rightarrow T(n)
$$
of GMTS. The good parameters of $T(i)$ are $\{\overline x(i)_{k,l}\}$ as determined by the preceding  GMTS $T(i-1)\rightarrow T(i)$, and a good change of parameters $\{x(i)_{k,l}\}$.

\section{Embedded resolution by Perron transforms}

We continue to use the notation of Section \ref{SecPrem}, and to assume that $\mbox{char}(k)=0$.

In the following proof we use the fact that the sequences of monoidal transforms constructed in the algorithms of 
%\cite{E} and 
\cite{E1} are GMTSs. This fact follows from the proofs in these papers.  
%Theorem 4.3.1 \cite{E}, 
We have that \cite[Theorem 4.3]{E1},
explaining the construction of a sequence of monodial transforms (with m=1) from a given UTS (a uniformizing transformation sequence) is a special case of  the proof of Theorem 4.8 \cite{Ast}.   On line -6 from the bottom of page 79, in Step 3 of \cite{Ast}, it is explicitely stated and shown that  the sequence of monoidal transforms is a GMTSs.

\begin{Theorem}\label{Theorem1} Suppose that $T$ satisfies the conditions of Lemma \ref{Lemma1} and that $T$ has a very good system of parameters $\{\overline x_{i,j}\}$.
Suppose that $f\in T$.
Then there exists a sequence of GMTSs 
$$
T=T_0\rightarrow T_1\rightarrow \cdots \rightarrow T_m
$$
along $\nu$ such that 
%the  resulting sequence of good parameters in $T_m$ 
%has a very good change of parameters $x_{1,1}(m),x_{1,2}(m),\ldots$ such that  
$$
f=x_{1,1}(m)^{d_{1,1}}\cdots x(m)_{1,s_1}^{d_{1,s_1}}x(m)_{2,1}^{d_{2,1}}\cdots x(m)_{2,s_2}^{d_{2,s_2}}x(m)_{3,1}^{d_{3,1}}\cdots x(m)_{u,s_u}^{d_{u,s_u}}\gamma
$$
where $\gamma$ is a unit in $T_m$ and $d_{ij}\in\NN$ for all $i,j$.
\end{Theorem}

\begin{proof} Let $k=\min\{i\mid\nu_i(f)<\infty\}$. By 
%Theorem 4.3.12 (page 80) \cite{E} (or 
\cite[Theorem 4.14]{E1} (which is applicable when $\nu$ has arbitrary rank and $\nu_1(f)<\infty$, using some Perron transforms of type (3,m)) and %Theorem 5.2.4 (page 92) \cite{E} (or 
\cite[Theorem 5.6]{E1}, there exists a sequence of GMTSs along $\nu_k$
$$
R_k=T_{P_{T,k}}\rightarrow R_k(1)\rightarrow \cdots R_k(n_{k,1})
$$
such that $R_k(n_{k,1})$ has a very good change of parameters 
$$
x_k(n_{k,1})_{1,1},\ldots,x_k(n_{k,1})_{1,s_k},\ldots
$$
such that $f=x_k(n_{k,1})_{1,1}^{d_{k,1}}\cdots x_k(n_{k,1})^{d_{k,s_k}}\gamma_k$ where $d_{k,1},\ldots,d_{k,s_k}\in \NN$ and $\gamma_k\in R_k(n_{k,1})$ is a unit.

Now by 
%Lemma 5.2.1 (page 89) \cite{E}, Remark 5.2.2 (page 91) \cite{E} and Lemma 5.2.3 (page 91) \cite{E} (or 
\cite[Lemma 5.3]{E1} (on line 11 of the statement of the lemma it should be 
``$(T_1)_{P_{T_1}^1}=R_1$''), \cite[Remark 5.4]{E1} and \cite[Lemma 5.5]{E1}, there exists a sequence of GMTSs along $\nu_{k-1}$,
$$
R_{k-1}=T_{P_{T,k-1}}\rightarrow R_{k-1}(1)\rightarrow \cdots \rightarrow R_{k-1}(n_{k-1,1})
$$
such that $R_{k-1}(n_{k-1,1})_{P_{R_{k-1}(n_{k-1,1}),k}}=R_k(n_{k,i})$ and $R_{k-1}(n_{k-1,1})$ has a very good change of parameters 
$$
x_{k-1}(n_{k-1,1})_{1,1},\ldots,x_{k-1}(n_{k-1,1})_{1,s_{k-1}},x_{k-1}(n_{k-1,1})_{2,1},\ldots,x_{k-1}(n_{k-1,1})_{2,s_{k}},\ldots
$$
such that
$$
f=x_{k-1}(n_{k-1,1})_{2,1}^{d_{k,1}}\cdots x_{k-1}(n_{k-1,1})_{2,s_k}^{d_{k,s_k}}\overline\gamma_k
$$
where $\overline \gamma_k\in R_{k-1}(n_{k-1,1})$ satisfies $\nu_k(\overline\gamma_k)=0$.

%Applying Theorem 4.3.12 \cite{E} (Theorem 4.14 \cite{E1}) and Theorem 5.2.4 \cite{E} (Theorem 5.6 \cite{E1}) to $R_{k-1}(n_{k-1,1})$ and $\overline \gamma_k$. we construct a sequence of GMTSs along $\nu_{k-1}$ such that $R_{k-1}(n_{k-1,2})$ has a very good change of parameters 
%$$
%x_{k-1}(n_{k-1,2})_{1,1},\ldots,x_{k-1}(n_{k-1,1})_{1,s_{k-1}},\ldots,x_{k-1}(n_{k-1,1})_{2,1},\ldots,x_{k-1}(n_{k-1,1})_{2,s_k,1},\ldots,
%$$
%such that
%$$
%f=x_{k-1}(n_{k-1,2})_{1,1}^{d_{k-1,1}}\cdots x_{k-1}(n_{k-1,2})_{1,s_{k-1}}^{d_{k-1,s_{k-1}}}
%x_{k-1}(n_{k-1,2})_{2,1}^{d_{k,1}}\cdots x_{k-1}(n_{k-1,2})_{2,s_k}^{d_{k,s_k}}\gamma_{k-1}
%$$
%where all $d_{ij}\in \NN$ and $\gamma_{k-1}\in R_k(n_{k-1,2})$ is a unit. 

By descending induction on $k$, successively applying 
%Theorem 4.3.12 \cite{E}, Theorem 5.2.4 \cite{E} and then Lemma 5.2.1 \cite{E}, Remark 5.2.2 \cite{E} and Lemma 5.2.3 \cite{E} (
\cite[Theorem 4.14]{E1}, \cite[Theorem 5.6]{E1} and then  \cite[Lemma 5.3]{E1}, \cite[Remark 5.4]{E1} and \cite[Lemma 5.5]{E1}, we obtain the conclusions of the theorem.

\end{proof}

\section{Monomial extensions}

Suppose that $K\rightarrow L$ is a finite extension of algebraic function fields over a field $k$ of characteristic zero, $\nu$ is a valuation of $K/k$ and $\omega$ is an extension of $\nu$ to $L$. 
Let 
$$
e=e(\omega|\nu)=[\Gamma_{\omega}:\Gamma_{\nu}], f=f(\omega|\nu)=[V_{\omega}/m_{\omega}:V_{\nu}/m_{\nu}].
$$

Suppose that $R\rightarrow S$ is an extension of algebraic regular local rings of $K$ and $L$ respectively such that $\omega$ dominates $S$ and $S$ dominates $R$. 
Suppose that $\{z_{i}\}$ are regular parameters in $R$ and $\{w_{j}\}$ are regular parameters in $S$. We will say that $R\rightarrow S$ is locally monomial (with respect to these systems of parameters) if there exists an $n\times n$ matrix $C=(c_{ij})$ with coefficients in $\NN$, where $n=\dim R=\dim S$ with ${\rm Det}(C)\ne 0$,
and units $\alpha_{i}\in S$ such that   
\begin{equation}\label{eq20}
z_{i}=\prod_{j=1}^n w_{j}^{c_{ij}}\alpha_{j}\mbox{ for }1\le i\le n.
\end{equation}

The following theorem is the local monomialization theorem proved in \cite[Theorem 1.1]{Ast}.

\begin{Theorem}\label{Theorem3} Suppose that $R^*\rightarrow S^*$ is an extension of algebraic regular local rings of $K$ and $L$ respectively such that $\omega$ dominates $S^*$ and $S^*$ dominates $R^*$. 
There exist sequences of monoidal transforms $R^*\rightarrow R$  and $S^*\rightarrow S$ along $\omega$ such that  $S$ dominates $R$ and $R\rightarrow S$ is locally monomial.  We can further construct $R\rightarrow S$ so that the regular parameters $z_i$ in $R$ and $w_j$ in $S$ giving the monomial form are very good parameters in $R$ and $S$ respectively. 
\end{Theorem}

With the notation of the conclusion of Theorem \ref{Theorem3},  we have by \cite[Theorem 4.2]{CP} that there exists a birational extension $R\rightarrow \overline R$ such that $\overline R$ is normal, $S$ dominates $R$  and $S$ is a localization of the integral closure of $\overline R$ in $L$. 

%By \cite[Theorem 4.8]{CP}, We can further obtain in the conclusions of Theorem \ref{Theorem3} that the very good parameters $x_{ij}$ and $y_{ij}$  giving a monomial form satisfy the conditions that 

%x_{k,i}=\left\{\begin{array}{ll}
%y_{k,1}^{c_{i,1}^k}y_{k,2}^{c_{i,2}^k}\cdots y_{k,s_k}^{c_{i,s_k}^k}\phi_{k,i}\alpha_{k,i}
%&\mbox{ if }1\le i\le s_k\\
%y_{k,i}&\mbox{ otherwise.}
%\end{array}\right.
%\end{equation}
%where $\phi_{k,i}$ are monomials in 
%$$
%y_{11},\ldots,y_{1s_1},y_{2,1},\ldots,y_{2,s_2},y_{3,1},\ldots,y_{k-1,s_{k-1}},
%$$
%$\alpha_{k,i}$ are units in $S$, and $\mbox{Det}(c_{i,1}^k)\ne 0$ for $1\le l\le u$.
%We will say that $R\rightarrow S$ has an $S$-good monomial form with respect to these systems of parameters if  $x_{1,1},\ldots$ are $S$-good parameters in $R$, $y_{1,1},\ldots$ are $S$-good parameters in $S$ and $R\rightarrow S$ has a good  monomial form with respect to these systems of parameters.

The following proposition is Theorem 6.1 \cite{CP}.

\begin{Proposition}\label{Prop1} Let $g_{1},\ldots,g_{f}$ be a basis of $V_{\omega}/m_{\omega}$ over $V_{\nu}/m_{\nu}$.  Then there exist an algebraic local ring $R'$ of $K$ which is dominated by $V_{\nu}$ such that whenever $R\rightarrow S$ is  an extension such that $R$ is an algebraic regular local ring of $K$ and $S$ is an algebraic regular local ring of $L$ which is dominated by $\omega$ and dominates $R$ such that $R\rightarrow S$ is locally monomial with regular parameters satisfying (\ref{eq20}), then
 $[S/m_S:R/m_R]=f$, $|{\rm Det}(C)|=e$, $[{\rm QF}(\hat S):{\rm QF}(\hat R)]=e$ and $g_{1},\ldots,g_{f}$ is a basis of $S/m_S$ over $R/m_R$.
 \end{Proposition}

\begin{Lemma}\label{Lemma2}  Suppose that $R$ and $S$ satisfy the conclusions of Lemma \ref{Lemma1} for $\nu$ and $\omega$ respectively,
$R\rightarrow S$ is a locally monomial extension and $f\in S$. Then there exists $g\in \overline R$ such that $f$ divides $g$ in $S$.
\end{Lemma}

\begin{proof}
Let $F$ be a Galois closure of $L$ over $K$ and let $\overline S$ be the integral closure of $R$ in $L$ and $\overline T$ be the integral closure of 
$\overline R$ in $F$. There exists $\gamma\in \overline S$ such that $\gamma$ is a unit in $S$ and $\gamma f\in \overline S$. Let $\overline f=\gamma f$. Let $G=\mbox{Gal}(F/K)$ and $g=\prod_{\sigma\in G}\sigma(\overline f)$. Then $\sigma(g)=g$ for all $\sigma\in G$ so that $g\in K$.   Further,
$\sigma(\overline f)\in \overline T$ for all $\sigma\in G$ so $g$ is integral over $\overline R$. Thus $g\in \overline R$ since $\overline R$ is normal. Let $h=\frac{g}{\overline f}$. Then $h\in L$ and $h\in \overline T$ so $h$ is integral over $\overline S$. Thus $h\in \overline S$ since $\overline S$ is normal, and so $\gamma h\in \overline S\subset S$. Thus $f$ divides $g$ in $S$.
\end{proof}

\section{Analysis when $\epsilon(\omega|\nu)=e(\omega|\nu)$}

In this section, let $K$ be an algebraic function field over a field $k$ of characteristic zero and let $\nu$ be a valuation of $K/k$ ($\nu$ is trivial on $k$). Assume that $L$ is a finite extension of $K$ and $\omega$ is an extension of $\nu$ to $L$. The ramification index $e(\omega|\nu)$ and initial index $\epsilon(\omega|\nu)$ are defined in Section \ref{SecInt}.
Let
$$
e=e(\omega|\nu)\mbox{ and }\epsilon=\epsilon(\omega|\nu).
$$

The following proposition is Proposition 3.7 \cite{CN}. It holds very generally for finite extensions of valued fields $(K,\nu)\rightarrow (L,\omega)$.

\begin{Proposition}\label{PropGS} Suppose that $K$ is a field, $\nu$ is a valuation of $K$, $L$ is a finite extension field of $K$ and  $\omega$ is an extension of $\nu$ to $L$ such that 
$$
1<\epsilon(\omega|\nu)=e(\omega|\nu).
$$
Let $\Gamma_{\nu,1}$ be the first convex subgroup of $\Gamma_{\nu}$ and $\Gamma_{\omega,1}$ be the first convex subgroup of $\Gamma_{\omega}$. Then $\Gamma_{\omega,1}\cong \ZZ$ and in the short exact sequence of groups
\begin{equation}\label{eqAb3}
0\rightarrow \Gamma_{\omega,1}/\Gamma_{\nu,1}\rightarrow \Gamma_{\omega}/\Gamma_{\nu}\rightarrow (\Gamma_{\omega}/\Gamma_{\omega,1})/(\Gamma_{\nu}/\Gamma_{\nu,1})\rightarrow 0
\end{equation}
we have that 
$$
(\Gamma_{\omega}/\Gamma_{\omega,1})/(\Gamma_{\nu}/\Gamma_{\nu,1})=0
$$
and 
$$
\Gamma_{\omega}/\Gamma_{\nu}\cong \Gamma_{\omega,1}/\Gamma_{\nu,1}\cong \ZZ_e.
$$
\end{Proposition}

The following proposition generalizes Proposition 7.4 of \cite{CN} from Abhyankar valuations on algebraic function fields to arbitrary valuations on characteristic zero algebraic function fields.

\begin{Proposition}\label{PropGoodForm} Suppose that $e(\omega|\nu)=\epsilon(\omega|\nu)$. Then there exist algebraic regular local rings $R$ of $K$ and $S$ of $L$ which are dominated by $\omega$ and $\nu$ respectively such that $S$ dominates $R$, $R$ dominates the ring $R'$ of Proposition \ref{Prop1} and $R$ has  good regular parameters $\{x_{i,j}\}$ and $S$ has  good regular parameters $\{y_{i,j}\}$ such that there is an expression
$$
x_{1,1}=\gamma y_{1,1}^e\mbox{ and }
x_{i,j}=y_{i,j}\mbox{ if $i>1$ or $j\ge 2$}
$$
where $\gamma$ is a unit in $S$. Further, if $e>1$, then $\nu(x_{1,1})$ is a generator of $\Gamma_{\nu,1}$ and $\omega(y_{1,1})$ is a generator of $\Gamma_{\omega,1}$. If $e=1$, then $\gamma=1$.
\end{Proposition}

\begin{Remark}\label{RemarkGoodForm} We can assume that the parameters $\{x_{ij}\}$ and $\{y_{ij}\}$ are very good parameters in the conclusions of Proposition \ref{PropGoodForm}.
\end{Remark}

\begin{proof} By  Theorem \ref{Theorem3} and Proposition \ref{Prop1} there exist algebraic regular local rings $R_0$ of $K$ and $S_0$ of $L$ such that 
 $\omega$ dominates $S_0$, $S_0$ dominates $R_0$, $R_0$ has  very good  parameters $\{x_{ij}\}$ and $S_0$ has very good  parameters $\{y_{ij}\}$ satisfying the conclusions of Proposition \ref{Prop1}.
 
 We reindex the very good parameters $\{x_{ij}\}$ and $\{y_{ij}\}$ by
 $$
 x_j=x_{l,i}\mbox{ if $j=t_1+\cdots+t_l+i$ with $1\le i\le t_{l+1}$}
 $$
 and
 $$
 y_j=y_{l,i}\mbox{ if $j=t_1+\cdots+t_l+i$ with $1\le i\le t_{l+1}$}.
 $$ 
 These parameters  have a monomial form
 \begin{equation}\label{eq23}
 x_j=\gamma_jy_1^{c_{j1}}\cdots y_n^{c_{jn}}\mbox{ for $1\le j\le n$}
 \end{equation}
 where $C=(c_{ij})$ is an $n\times n$ matrix with ${\rm Det}(C)\ne 0$ and $\gamma_j$ are units in $S_0$. By the proof of \cite[Theorem 4.10]{CP}, we can assume that
 \begin{equation}\label{eq24}
 \begin{array}{lll}
 \Gamma_{\omega}/\Gamma_{\nu}&\cong& (\sum \omega(y_{i})\ZZ)/(\sum \nu(x_{i})\ZZ)\\
 &\cong& \ZZ^n/C^t\ZZ^n
 \end{array}
 \end{equation}
 so that
 \begin{equation}\label{eq22}
 |{\rm Det}(C)|=e.
 \end{equation}
 First suppose that $\epsilon(\omega|\nu)=e(\omega|\nu)>1$. Then $\Gamma_{\omega,1}\cong \ZZ$ and $\Gamma_{\omega,1}/\Gamma_{\nu,1}\cong \ZZ_e$ by Proposition \ref{PropGS}. In particular, $s_1=1$.
 
 By \cite[Theorem 4.8]{CP}, we may assume that
 $$
 x_1=\gamma y_1^{c_{11}}, x_2=y_2,\ldots,x_{t_1}=y_{t_1}
 $$
 where $\gamma$ is a unit in $S_0$. Then from (\ref{eq24}), we have that $c_{11}=e$ and $|{\rm Det}(\overline C)|=1$ where
 $$
 \overline C=\left(\begin{array}{ccc}
 c_{t_1+1,t_1+1}&\cdots&c_{t_1+1,n}\\
 \vdots\\
 c_{n,t_1+1}&\cdots&c_{n,n}
 \end{array}\right).
 $$
 We define a birational extension along $\nu$, $R_0\rightarrow R_1=R_0[x(1)_1,\ldots,x(1)_n]_{\nu}$ by
 $$
 x_j=\left\{\begin{array}{ll}
 x(1)_j&\mbox{ for }1\le j\le t_1\\
 x(1)_jx(1)_2^{c_{j2}}\cdots x(1)_{t_1}^{c_{j,t_1}}&\mbox{ for } t_1<j\le n
  \end{array}\right.
 $$
  to get that $S_0$ dominates $R_1$ and $R_1\rightarrow S_0$ is locally monomial with
 $$
 x(1)_j=\left\{\begin{array}{ll}
 \gamma y_1^e&\mbox{ if }j=1\\
 y_j&\mbox{ if }1<j\le t_1\\
 y_1^{c_{j1}}y_{t_1+1}^{c_{j,t_1+1}}\cdots y_n^{c_{j,n}}\alpha_j&\mbox{ if }t_1+1\le j\le n
 \end{array}\right.
 $$
 where $\alpha_j\in S_0$ are units.

 Since ${\rm Det}(\overline C)=\pm 1$, there exist $r_{t_1+1},\ldots,r_n\in \ZZ$ such that 
 $$
 \overline C\left(\begin{array}{c}r_{t_1+1}\\ \vdots\\ r_n\end{array}\right)=-\left(\begin{array}{c}
 c_{t_1+1,1}\\ \vdots \\c_{n,1}\end{array}\right).
 $$
 Define $d_{t_1+1},\ldots,d_n\in \NN$ by
 $$
 \left(\begin{array}{c} d_{t_1+1}\\ \vdots\\ d_n\end{array}\right)=\overline C
 \left(\begin{array}{c} 1\\ \vdots\\ 1\end{array}\right).
 $$
 There exists $v\in \ZZ_{>0}$ such that $r_i+ve>0$ for all $i$. Perform the sequence of GMTSs 
 of type (3,1) $S_0\rightarrow S_1$  along $\omega$ where $S_1$ has good parameters $\{y(1)_i\}$ defined by  
 $$
 y_i=\left\{\begin{array}{ll}
 y(1)_i&\mbox{ if }1\le i\le t_1\\
 y(1)_iy(1)_1^{r_i+ve}&\mbox{ if }t_1<i\le n.
 \end{array}\right.
 $$
 We have that $S_1=S_0[y_1(1),\ldots,y_n(1)]_{\omega}$  dominates $R_1$ and $R_1\rightarrow S_1$ is locally monomial. There exist units $\gamma_i'\in S_1$  and $g_i\in\NN$ such that 
 $$
 x(1)_i=\left\{\begin{array}{ll}
 \gamma y(1)_1^e&\mbox{ if }i=1\\
 y(1)_i&\mbox{ if }1<i\le t_1\\
 \gamma_i'y(1)_1^{evg_i}y(1)_{t_1+1}^{c_{i,t_1+1}}\cdots y(1)_n^{c_{i,n}}&\mbox{ if }t_1< i\le n.
 \end{array}\right.
 $$

 Now perform the sequence of GMTSs $R_1\rightarrow R_2$ of type (3,1) along $\nu$ defined by
 $$
 x(1)_i=\left\{\begin{array}{ll} x(2)_1&\mbox{ if }1\le i\le t_1\\
 x(2)_1^{vg_i}x(2)_i&\mbox{ if }t_1<i\le n.
 \end{array}\right.
 $$
 Then  $S_1$ dominates $R_2$ and there exist units $\gamma(1)_i\in S_1$ such that
 \begin{equation}\label{eq21}
 x(2)_i=\left\{\begin{array}{ll}
 \gamma y(1)_1^e&\mbox{ if }i=1\\
 y(1)_i&\mbox{ if }1<i\le t_1\\
 \gamma(1)_iy(1)_{t_1+1}^{c_{i,t_1+1}}\cdots y(1)_n^{c_{i,n}}&\mbox{ if }t_1<i\le n.
 \end{array}\right.
 \end{equation}
   Let $B=\overline C^{-1}$. Write
 $$
 B=\left(\begin{array}{lll}
 b_{t_1+1,t_1+1}&\cdots&b_{t_1+1,n}\\
 &\vdots&\\
 b_{n,t_1+1}&\cdots&b_{n,n}
 \end{array}
 \right)
 $$
 with $b_{i,j}\in \ZZ$. We now replace the $y(1)_i$ with the product of the unit
 $\gamma(1)_{t_1+1}^{-b_{i,t_1+1}}\cdots \gamma(1)_n^{-b_{i,n}}$ and $y(1)_i$ for $t_1+1\le i\le n$ to get $\gamma_i(1)=1$ for $t_1+1\le i\le n$ in (\ref{eq21}).
 
Now define a birational transformation $R_2\rightarrow R_3$ along $\nu$ by $R_3=R_2[x(3)_{t_1+1},\ldots,x(3)_n]_{\nu}$ where $R_3$ has regular parameters $\{x(3)_i\}$ defined by
$$
x(3)_i=\left\{\begin{array}{ll}
x(2)_i&\mbox{ if }1\le i\le t_1\\
x(2)_{t_1+1}^{c_{i,t_1+1}}\cdots x(2)_n^{c_{i,n}}&\mbox{ for }t_1+1\le i\le n.
\end{array}\right.
$$ 
The ring $R_3$ is a regular local ring with regular parameters $x(3)_1,\ldots,x(3)_n$. We have that $S_1$ dominates $R_3$ and
$$
x(3)_i=\left\{\begin{array}{ll}
\gamma y(1)_1^e&\mbox{ if }i=1\\
y(1)_i&\mbox{ if }2\le i\le n
\end{array}\right.
$$
where $\gamma$ is a unit in $S_1$.  Going back to (\ref{eq24}), we see that $\omega(y(1)_1)$ is a generator of $\Gamma_{\omega,1}$ and $\nu(x(3)_1)$ is a generator of $\Gamma_{\nu,1}$. 
We thus have the conclusions of the proposition. 

Now suppose that $e=1$. This case is much simpler.  In (\ref{eq22}) we then have that ${\rm Det}(C)=\pm 1$. Taking $B=C^{-1}=(b_{i,j})$, we can then make the change of variables in $S_0$ 
replacing the $y_i$ with the product of the unit
 $\gamma_1^{-b_{i,1}}\cdots \gamma_n^{-b_{i,n}}$ times $y_i$ for $1\le i\le n$ to get $\gamma_i=1$ for $1\le i\le n$ in (\ref{eq23}).
 
 Now define a birational transformation $R_0\rightarrow R_1$ along $\nu$ by $R_1=R_0[x(1)_1,\ldots,x(1)_n]_{\nu}$ where
$$
x_i=x(1)_1^{c_{i,1}}\cdots x(1)_n^{c_{i,n}}\mbox{ for }1\le i\le n.
$$ 
The ring $R_1$ is a regular local ring with regular parameters $x(1)_1,\ldots,x(1)_n$. We have that $R_1$ is dominated by $S$, and
$$
x(1)_i = y(1)_i\mbox{ for }1\le i\le n,
$$
giving the  conclusions of the proposition.
\end{proof}

\begin{Proposition}\label{PropSR} Suppose that $e(\omega|\nu)=\epsilon(\omega|\nu)$ and 
$R_0\rightarrow S_0$ has the form of the conclusions of Proposition \ref{PropGoodForm} for 
 good  parameters
$\{\overline x(0)_{i,j}\}$ in $R_0$ and good parameters $\{\overline y(0)_{i,j}\}$ in $S_0$.
Then there exist $z_1,\ldots,z_m\in V_\omega$ such that $S_0=R_0[z_1,\ldots,z_m]_\omega$.

% and $i$ and $j$ are such that $1\le i<j\le n$ (with $\nu(x_i)<\nu(x_j)$). 

Let $R_0\rightarrow R_1$ be a GMTS along $\nu$,
constructed from a good change of parameters  $\{x(0)_{ij}\}$ in $R_0$, and giving  good parameters $\{\overline x(1)_{ij}\}$ in $R_1$. 

We then have a good change of parameters 
$\{y(0)_{ij}\}$ in $S_0$ defined by
$$
y(0)_{ij}=\left\{\begin{array}{ll}
\overline y(0)_{ij}\mbox{ if }i=1 \mbox{ and } j=1\\
x(0)_{ij}\mbox{ otherwise. }
\end{array}\right.
$$
The good parameters $\{x(0)_{ij}\}$ and $\{y(0)_{ij}\}$ continue to have the form of the conclusions of Proposition \ref{PropGoodForm}.

There exists a GMTS  $S_0\rightarrow S_1$ along $\omega$,
constructed from the above good  parameters  $\{y(0)_{ij}\}$ in $S_0$, and giving  good parameters $\{\overline y(1)_{ij}\}$ in $S_1$, 
such that there is a good change of parameters $\{y'_{i,j}\}$ in  $S_1$ such that 
$\{x(1)_{i,j}\}$ and $\{y'_{i,j}\}$ are related by an expression of the form of the conclusions of Proposition \ref{PropGoodForm} and we have that 
 $S_1=R_1[z_1,\ldots,z_m]_{\omega}$.
\end{Proposition}

\begin{proof} The expression $S_0=R_0[z_1,\ldots,z_m]_\omega$ follows since $S_0$ is essentially of finite type over $k$.

 Suppose that the GMTS $R_0\rightarrow R_1$ is of type (1,m). Then $R_1=R[N_1,\ldots,N_{s_m}]_{\nu}$ where 
$$
x(0)_{mj}=N_1^{a_{j1}}\cdots N_{s_m}^{a_{js_m}}\mbox{ for  $1\le j\le s_m$}
$$
and the good regular parameters $\overline x(1)_{ij}$ in $R_1$ are defined by
$$
\overline x(1)_{ij}=
\left\{\begin{array}{ll}
N_j&\mbox{ if $i=m$ and $1\le j\le s_m$}\\
x(0)_{ij}&\mbox{ otherwise}
\end{array}\right.
$$

If $e>1$ then  $\Gamma_{\nu,1}$ has rational rank 1 ($s_1=1$), and we then cannot perform a GMTS of type (1,1).  In particular, we have $e=1$ if $m=1$. 

 We define the  Perron transform $S_0\rightarrow S_1[N_1,\ldots,N_{s_m}]_{\omega}$ 
of type (1,m), giving good regular parameters
$\overline y(1)_{ij}$ such that 
$$
\overline y(1)_{ij}=
\left\{\begin{array}{ll}
N_j=\overline x(1)_{ij}&\mbox{ if $i=m$ and $1\le j\le s_m$}\\
y(1)_{ij}&\mbox{ otherwise.}
\end{array}\right.
$$
We then have that the good parameters $\{\overline x(1)_{ij}\}$ and $\{\overline y(1)_{ij}\}$ are related by an expression of the form of Proposition \ref{PropGoodForm} and we have that 
 $S_1=R_1[z_1,\ldots,z_m]_{\omega}$. 
 
 Suppose that the GMTS $R_0\rightarrow R_1$ is of type (2,m). Then $R_1=R_0[N_1,\ldots N_{s_m},N_r]_{\nu}$ where $s_m<r\le t_m$,
 $$
x(0)_{mj}=\left\{\begin{array}{ll}
N_1^{a_{j1}}\cdots N_{s_m}^{a_{js_m}}N_r^{a_{j,s_m+1}}&\mbox{ if  $1\le j\le s_m$}\\
N_1^{a_{s_m+1,1}}\cdots N_{s_m}^{a_{s_m+1,s_m}}N_r^{a_{s_m+1,s_m+1}}&\mbox{ if  $ j= r$}
\end{array}\right.
$$
where $a_{ij}\in \NN$, $\mbox{Det}(a_{ij})=\pm 1$ and $\nu(N_1),\ldots,\nu(N_{s_m})>0$ and $\nu_m(N_r)=0, \nu(N_r) \ge 0$. 

Let $\mathfrak m = m_{\nu}\cap R_0[N_1,\ldots,N_{s_m},N_r]$. Choose  $y\in R_1$  such that
$y$ is the lift to $R_1$ of a generator of the maximal ideal of 
\begin{equation}\label{eq11}
\begin{array}{ll}
&R_1/(x(0)_{1,1},\ldots,x(0)_{m-1,t_{m-1}},N_1,\ldots,N_{s_m},x(0)_{m,s_m+1},\ldots ,x(0)_{m,r-1},x(0)_{m,r+1},\ldots )\\
\cong &(R_0/m_{R_0})[N_r]_{\mathfrak m(R_0/m_{R_0})[N_r]}.
\end{array}
\end{equation}
 
 Let $\lambda$ be the  smallest natural number  such that $y\in P_{R_1,\lambda}\setminus P_{R_1,\lambda+1}$.  Then the  regular parameters $\{\overline x(1)_{ij}\}$ in $R_1$ are defined by

$$
 \overline x(1)_{ij}=\left\{\begin{array}{ll}
 N_j&\mbox{ if $i=m$ and $1\le j\le s_m$}\\
 y&\mbox{ if  $i=m$ and $j=r$}\\
 x(0)_{ij}&\mbox{ otherwise}
 \end{array}\right.
 $$
 if $\lambda=m$, and

 $$
 \overline x(1)_{ij}=\left\{\begin{array}{ll}
 N_j&\mbox{ if $i=m$ and $1\le j\le s_m$}\\
 y&\mbox{ if $i=\lambda$ and $j=t_{\lambda}+1$}\\
 x(0)_{i,j-1}&\mbox{ if  $i=m$ and $j\ge r+1$}\\
 x(0)_{ij}&\mbox{ otherwise}
 \end{array}\right.
 $$
 if $\lambda\ne m$. 

First suppose that $m>1$ or  $m=e=1$.

We define the  Perron transform $S_0\rightarrow S_1=S_0[N_1,\ldots,N_{s_m},N_r]_{\omega}$ 
of type (2,m).
 Let $\mathfrak n = m_{\omega}\cap S_0[N_1,\ldots,N_{s_m},N_r]$. 
 We have 
\begin{equation}\label{eq12}
\begin{array}{ll}
&S_1/(y(0)_{1,1},\ldots,y(0)_{m-1,t_{m-1}},N_1,\ldots,N_{s_m},y(0)_{m,s_m+1},\ldots ,y(0)_{m,r-1},y(0)_{m,r+1},\ldots )\\
\cong &(S_0/m_{S_0})[N_r]_{\mathfrak n(S_0/m_{S_0})[N_r]}.
\end{array}
\end{equation}
 By  (\ref{eq11}) and (\ref{eq12}), the dominant homomorphism $R_1\rightarrow S_1$,  induces a dominant homomorphism
 $$
 (R_0/m_{R_0})[N_r]_{\mathfrak m(R_0/m_{R_0})[N_r]}\rightarrow(S_0/m_{S_0})[N_r]_{\mathfrak n(S_0/m_{S_0})[N_r]} .
 $$ 
 
 Suppose that $y$ is the lift of $\overline y\in  (R_0/m_{R_0})[N_r]_{\mathfrak m(R_0/m_{R_0})[N_r]}$. We can assume that $\overline y\in   (R_0/m_{R_0})[N_r]$ is irreducible. Then $\overline y$ is a separable polynomial in the polynomial ring $(R_0/m_{R_0})[N_r]$ since $R_0/m_{R_0}$ has characteristic zero. Thus $\overline y\in  \mathfrak n(S_0/m_{S_0})[N_r]\subset  (S_0/m_{S_0})[N_r]$ is separable and hence $\overline y$ is a generator of the maximal ideal of  
 $$
 (S_0/m_{S_0})[N_r]_{\mathfrak n(S_0/m_{S_0})[N_r]}.
 $$
 
We may thus define our good parameters $\overline y(1)_{ij}$ in $S_1$ by 
 
$$
 \overline y(1)_{ij}=\left\{\begin{array}{ll}
 N_j&\mbox{ if $i=m$ and $1\le j\le s_m$}\\
 y&\mbox{ if  $i=m$ and $j=r$}\\
 y(0)_{ij}&\mbox{ otherwise}
 \end{array}\right.
 $$
 if $\lambda=m$, and

$$
 \overline y(1)_{ij}=\left\{\begin{array}{ll}
 N_j&\mbox{ if $i=m$ and $1\le j\le s_m$}\\
 y&\mbox{ if $i=\lambda$ and $j=t_{\lambda}+1$}\\
 y(0)_{i,j-1}&\mbox{ if  $i=m$ and $j\ge r+1$}\\
 y(0)_{i,j}&\mbox{ otherwise}
 \end{array}\right.
 $$
 if $\lambda\ne m$. 

 We  have that the good parameters $\{\overline x(1)_{ij}\}$ in $R_1$ and $\{\overline y(1)_{ij}\}$ in $S_1$ are related by an expression of the form of Proposition \ref{PropGoodForm} and we have that 
 $S_1=R_1[z_1,\ldots,z_m]_{\omega}$. 
 
 Now suppose that $m=1$ and $e>1$. Then $x(0)_{1,1}=\gamma y(0)_{1,1}^e$ and $x(0)_{1,r}=y(0)_{1,1}$. We then have that $s_1=1$, $\nu(x(0)_{11}$ is a generator of $\Gamma_{\nu,1}\cong\ZZ$ and $\nu(y(0)_{1,r})$ is a generator of $\Gamma_{\omega,1}\cong\ZZ$. Thus there exists $a\in \ZZ_+$ such that $\nu(x(0)_{1r})=a\nu(x(0)_{11})$ and so the equations defining $R_0\rightarrow R_1$ are $x(0)_{11}=N_1$ and $x(0)_{1r}=N_1^aN_r$. We have that $\omega(y(0)_{1r})=ea\omega(y(0)_{11})$. Define a GMTS along $\omega$,  $S_0\rightarrow S_1$, by $S_1=S_0[M_1,M_r]_{\omega}$ where $y(0)_{1,1}=M_1, y(0)_{1r}=M_1^{ea}M_r$. We have that $M_r=\gamma^aN_r$ so $S_1=R_1[z_1,\ldots,z_m]_{\omega}$.  As in the case $m>1$ or $m=e=1$, we may define our good parameters $\{\overline y(1)_{ij}\}$ in $S_1$ so that the good parameters $\{\overline x(1)_{ij}\}$ in $R_1$ and the good parameters $\{\overline y(1)_{ij}\}$ in $S_1$ are related by an expression of the form of Proposition \ref{PropGoodForm}.

Suppose that the MTS $R_0\rightarrow R_1$ is of the type (3,m). 
Then $R_1=R[N]_{\nu}$ where 
$$
N=\frac{x(0)_{kl}}{x(0)_{m1}^{d_1}\cdots x(0)_{ms_1}^{d_{s_m}}}.
$$
for some  $d_1,\ldots,d_{s_m}\in\NN$,  $k>m$ and $1\le l\le t_k$. 
The good parameters $\{\overline x(1)_{1,1}\}$ in $R_1$ are defined by
$$
\overline x(1)_{ij}=\left\{\begin{array}{ll}
N&\mbox{ if $i=k$ and $j=l$}\\
x(0)_{ij}&\mbox{ otherwise.}
\end{array}\right.
$$
If $m>1$ or $e=m=1$, then
$$
N=\frac{y(0)_{k,l}}{y(0)_{m1}^{d_1}\cdots y(0)_{m,s_m}^{d_{s_m}}}
$$
and if $m=1$ and $e>1$, then
$$
N=\gamma^{-d_1}\frac{y(0)_{k,l}}{y(0)_{11}^{ed_1}}.
$$
We may thus define a  GMTS along $\omega$,  $S_0\rightarrow S_1$, of type (3,m) by $S_1=S_0[N]_{\omega}$.  

The good parameters $\{\overline y(1)_{ij}\}$ in $S_1$ defined by the GMTS are such that  after making a good change of parameters, replacing $\overline y(1)_{k,l}$ with $x(1)_{k,l}$, the good parameters $\{\overline x(1)_{ij}\}$ in $R_1$ and the good parameters $\{\overline y(1)_{ij}\}$ in $S_1$ are related by an expression of the form of Proposition \ref{PropGoodForm}. We have that $S_1=R_1[z_1,\ldots,z_m]_{\omega}$.
\end{proof}

We now prove Theorem \ref{TheoremA} from Section \ref{SecInt}, which we restate in Theorem \ref{Theorem2}.

\begin{Theorem}\label{Theorem2} Let $K$ be an algebraic function field over a field $k$ of characteristic zero and let $\nu$ be a valuation of $K/k$ ($\nu$ is trivial on $k$). Assume that $L$ is a finite extension of $K$ and $\omega$ is an extension of $\nu$ to $L$. Then $V_{\omega}$ is essentially finitely generated over $V_{\nu}$
If and only if $e(\omega|\nu)=\epsilon(\omega|\nu)$.
\end{Theorem} 

\begin{proof} If $V_{\omega}$ is essentially finitely generated over $V_{\nu}$ then $e(\omega|\nu)=\epsilon(\omega|\nu)$ by Theorem 4.1 \cite{CN}.

Suppose that $e(\omega|\nu)=\epsilon(\omega|\nu)$. We will show that $V_{\omega}$ is essentially finitely generated over $V_{\nu}$. Let $R_0\rightarrow S_0$ be such that $R_0$ satisfies the conclusions of Lemma \ref{Lemma1} and the conclusions of Proposition \ref{PropGoodForm} with respect  to very good parameters in $R_0$ and very good parameters in $S_0$.
Write $S_0=R_0[z_1,\ldots,z_m]_{\omega}$.  We will show that $V_{\omega}=V_{\nu}[z_1,\ldots,z_n]_{\omega}$.

Suppose that $f\in V_{\omega}$.  Write $f=\frac{g}{h}$ with $g,h\in S_0$. By Lemma \ref{Lemma2}, and since with the conclusions of Proposition \ref{PropGoodForm} $S_0$ is a localization of the integral closure of $R_0$ in $L$, there exists $c\in R_0$ such that $gh$ divides $c$ in $S_0$. By Theorem \ref{Theorem1}, there exists a sequence of GMTSs $R_0\rightarrow R_1\rightarrow \cdots\rightarrow R_m$ such that
$$  
c=x(m)_{1,1}^{d_{1,1}}\cdots x(m)_{1,s_1}^{d_{1,s_1}}x(m)_{2,1}^{d_{2,1}}\cdots x(m)_{2,s_2}^{d_{2,s_2}}x(m)_{3,1}^{d_{3,1}}\cdots x(m)_{u,s_u}^{d_{u,s_u}}\gamma
$$
where $\gamma$ is a unit in $R_m$ and $d_{ij}\in\NN$ for all $i,j$.

By Proposition \ref{PropSR}, there exists a sequence of GMTSs $S_0\rightarrow S_1\rightarrow\cdots\rightarrow S_m$ such that $S_m=R_m[z_1,\ldots,z_n]_{\omega}$ and there are good parameters $\{y_{ij}(m)\}$ in $S_m$ such that 
$x(m)_{ij}= y(m)_{ij}$ if $1< i\le u$ and $1\le j\le s_i$ or if $i=1$ and $j>1$. We further have that
$$
 x(m)_{1,1}=\left\{\begin{array}{ll}
\tau y(m)_{1,1}^e\mbox{ where $\tau\in S_m$ is a unit }&\mbox{ if $e>1$}\\
y(m)_{1,1}&\mbox{ if $e=1$.}\end{array}\right.
$$
Thus in $S_m$, $c$ has an expression
$$
c=y(m)_{1,1}^{ed_{1,1}}y(m)_{1,2}^{d_{1,2}}\cdots y(m)_{1,s_1}^{d_{1,s_1}}y(m)_{2,1}^{d_{2,1}}\cdots y(m)_{2,s_2}^{d_{2,s_2}}x(m)_{3,1}^{d_{3,1}}\cdots y(m)_{u,s_u}^{d_{u,s_u}}\gamma'
$$
where $\gamma'$ is a unit in $S_m$ ($\gamma'=\gamma$ if $e=1$).

Since $gh$ divides $c$ in $S_m$, we have expressions 
$$
g=y(m)_{1,1}^{a_{1,1}}\cdots y(m)_{1,s_1}^{a_{1,s_1}}y(m)_{2,1}^{a_{2,1}}\cdots y(m)_{2,s_2}^{a_{2,s_2}}x(m)_{3,1}^{a_{3,1}}\cdots y(m)_{u,s_u}^{a_{u,s_u}}\alpha
$$
where $\alpha$ is a unit in $S_m$ and $a_{ij}\in\NN$ for all $i,j$, and 

$$
h=y(m)_{1,1}^{b_{1,1}}\cdots y(m)_{1,s_1}^{b_{1,s_1}}y(m)_{2,1}^{b_{2,1}}\cdots y(m)_{2,s_2}^{b_{2,s_2}}x(m)_{3,1}^{b_{3,1}}\cdots y(m)_{u,s_u}^{b_{u,s_u}}\beta
$$
where $\beta$ is a unit in $S_m$ and $b_{ij}\in\NN$ for all $i,j$. 
Let 
$$
W_1=x(m)_{1,1}^{a_{1,1}}x(m)_{1,2}^{ea_{1,2}}\cdots x(m)_{1,s_1}^{ea_{1,s_1}}x(m)_{2,1}^{ea_{2,1}}\cdots x(m)_{2,s_2}^{ea_{2,s_2}}x(m)_{3,1}^{ea_{3,1}}\cdots x(m)_{u,s_u}^{ea_{u,s_u}}
$$
and
$$
W_2=x(m)_{1,1}^{b_{1,1}}x(m)_{1,2}^{eb_{1,2}}\cdots x(m)_{1,s_1}^{eb_{1,s_1}}x(m)_{2,1}^{eb_{2,1}}\cdots x(m)_{2,s_2}^{eb_{2,s_2}}x(m)_{3,1}^{eb_{3,1}}\cdots x(m)_{u,s_u}^{eb_{u,s_u}}.
$$
We have that $\nu(W_1)=e\omega(g)>e\omega(h)=\nu(W_2)$.

By Proposition \ref{PropZa}, there exists a sequence of GMTSs 
$$
R_m\rightarrow R_{m+1}\rightarrow \cdots \rightarrow R_v
$$
 of types (1,m) and (3,m) such that $W_2$ divides $W_1$ in $R_v$.

By Proposition \ref{PropSR}, there exists a sequence of GMTSs $S_m\rightarrow S_{m+1}\rightarrow\cdots\rightarrow S_v$ such that $S_v=R_v[z_1,\ldots,z_n]_{\omega}$
and there exists a good change of parameters $\{y(v)_{ij}\}$ in $S_v$ such that the good parameters $\{x(v)_{ij}\}$ of $R_v$ have the good form of the conclusions of Proposition \ref{PropSR}. 
Thus $W_2$ divides $W_1$ in $S_v$. Now $W_1=g^e\alpha^{-e}\tau^{a_{1,1}}$ and $W_2=h^e\beta^{-e}\tau^{b_{1,1}}$. Thus $h^e$ divides $g^e$ in $S_v$. Now $g$ and $h$ are monomials in the good parameters of $S_v$ times  units. Thus $h$ divides $g$ in $S_v$ and so $f=\frac{g}{h}\in S_v$. Thus $f\in V_{\nu}[z_1,\ldots,z_m]_{\omega}$. Since this is true for all $f\in V_{\omega}$, we have that $V_{\omega}=V_{\nu}[z_1,\ldots,z_m]_{\omega}$ is essentially finitely generated over $V_{\nu}$.

\end{proof}

\end{document}